\documentclass[a4paper]{amsart}
\usepackage{amssymb}
\usepackage[hidelinks]{hyperref}
\usepackage{graphicx}
\usepackage{overpic}
\usepackage{breakcites}

\newcommand{\Z}{\mathbb{Z}}
\newcommand{\R}{\mathbb{R}}
\newcommand{\C}{\mathbb{C}}
\newcommand{\T}{\mathbb{T}}
\newcommand{\KK}{\bar{K}}
\newcommand{\EE}{\bar{E}}

\newcommand{\re}{\operatorname{Re}}
\newcommand{\im}{\operatorname{Im}}

\numberwithin{equation}{section}
\numberwithin{figure}{section}

\theoremstyle{plain} 
\newtheorem{theorem}{Theorem}[section]

\newtheorem{corollary}[theorem]{Corollary}
\newtheorem{proposition}[theorem]{Proposition}
\newtheorem{conjecture}[theorem]{Conjecture}
\theoremstyle{definition} 

\theoremstyle{remark}
\newtheorem{remark}[theorem]{Remark}

\def \cO{\mathcal O}
\def \cS{\mathcal S}

\def \tD {\mathrm{tD}}

\def \tDD {\mathrm{t\Delta}}
\def \tP {\mathrm{tP}}

\def \oP {\mathrm{oP}}
\def \oH {\mathrm{oH}}
\def \H {\mathrm{H}}

\def \oDD {\mathrm{o\Delta}}
\def \P {\mathrm P}

\newcommand{\cM}{\mathcal{M}}

\begin{document}

\title[The $\oH$ family]{An orthorhombic deformation family\\ of Schwarz' H surfaces}

\author{Hao Chen}
\address[Chen]{Georg-August-Universit\"at G\"ottingen, Institut f\"ur Numerische und Angewandte Mathematik}
\email{h.chen@math.uni-goettingen.de}
\thanks{H.\ Chen is supported by Individual Research Grant from Deutsche Forschungsgemeinschaft within the project ``Defects in Triply Periodic Minimal Surfaces'', Projektnummer 398759432.}

\author{Matthias Weber}
\address[Weber]{Indiana University, Department of Mathematics}
\email{matweber@indiana.edu}

\keywords{Triply periodic minimal surfaces}
\subjclass[2010]{Primary 53A10}

\date{\today}

\begin{abstract}
The classical H surfaces of H.\ A.\ Schwarz form a 1-parameter family of triply
periodic minimal surfaces (TPMS) that are usually described as close relatives
to his more famous P surface. However, a crucial distinction between these
surfaces is that the P surface belongs to a 5-dimensional smooth family of
embedded TPMS of genus three discovered by W.\ Meeks, while the H surfaces are
among the few known examples outside this family.  We construct a 2-parameter
family of embedded TPMS of genus three that contains the H family and meets the
Meeks family.  In particular, we prove that H surfaces can be deformed
continuously within the space of TPMS of genus three into Meeks surfaces.
\end{abstract}

\maketitle

\section{Introduction}

This is the second of two papers dealing with new 2-dimensional families of
embedded triply periodic minimal surfaces (TPMS) of genus three whose
1-dimensional ``intersections'' with the well-known Meeks family exhibit
singularities in the moduli space of TPMS.

Among the many TPMS discovered by H.\ A.\ Schwarz~\cite{schwarz1890} is a
1-parameter family $\H$ that can be constructed by extending the Plateau
solution for the boundaries of the two triangular faces of a triangular prism.
Such a Plateau solution does not exist for all heights of the prism.  For small
heights, there are two distinct solutions. One of them limits in the most
symmetric singly periodic Scherk surfaces with 6 annular ends. The other
degenerate to a foliation of $\R^3$ by horizontal parallel planes that are
joined by catenoidal necks, placed in a hexagonal lattice.

This family is remarkable because it does not belong to the 5-dimensional Meeks
family $\cM$ of TPMS of genus 3~\cite{meeks1990}.  Members of that family have
the eight branched values of the Gauss map forming four antipodal pairs, while
for an $\H$ surface, they are located at the north and south poles of the
2-sphere and the six vertices of a triangular prism.  The only other known TPMS
of genus 3 outside $\cM$ are the Gyroid-Lidinoid family~\cite{schoen1970,
lidin1990, fogden1993, fogden1999, weyhaupt2006, weyhaupt2008}, and the
recently discovered $\tDD$ family~\cite{chenweber1}.

We will exhibit a 2-parameter family $\oH$ of embedded TPMS of genus $3$ that
can be understood as an orthorhombic deformation family of Schwarz' $\H$
surfaces.  The closure of this family has a 1-dimensional intersection with the
$\oP$ family, a classical 2-parameter orthorhombic deformation family of
Schwarz $\P$ surface.  However, surfaces in $\oH$ are not in $\cM$.  This is,
after the $\oDD$ family in \cite{chenweber1}, another 2-parameter non-Meeks
family of TPMS of genus 3.

\medskip

\begin{figure}[hbt]
	\includegraphics[width=0.95\textwidth]{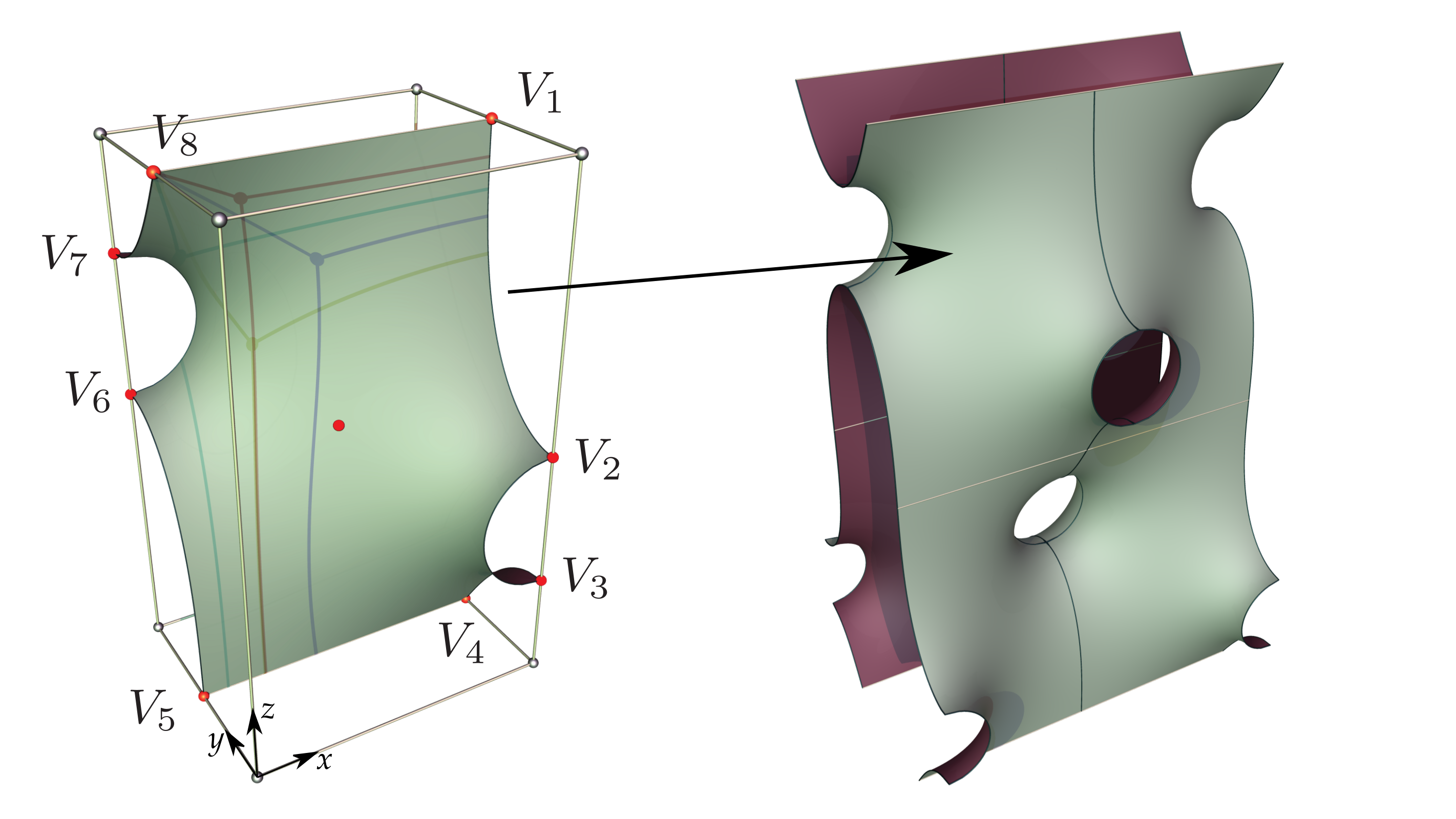}
	\caption{Fundamental Piece and Translational Fundamental Piece}
	\label{fig:fundpieces}
\end{figure}

Consider an embedded minimal surfaces $S$ inside an axes parallel box
$[-A,A]\times[-B,B]\times[-1,1]$ such that
\begin{itemize}
	\item $S$ satisfies free boundary condition on the vertical planes $x=\pm A$
		and $y=\pm B$, and fixed (Plateau) boundary condition on the horizontal
		segments $\{(x,0,\pm 1) \mid -A\le x \le A\}$.

	\item $S$ intersects the edges of the box in eight vertices, but disjoint
		from the vertical lines with $(x,y)=\pm(+A,+B)$.  Hence, apart from the
		four ends of the fixed boundaries, $S$ intersects the vertical lines
		$(x,y)=\pm(+A,-B)$ in two vertices each.

	\item $S$ is symmetric under the inversion in the origin.
\end{itemize}
Therefore $S$ is a right-angled minimal octagon, with its inversion center at
the origin.  See Figure \ref{fig:fundpieces} (left) for an example.  If the
vertices are labeled as in this figure, then the fixed boundaries are the
segments $V_8V_1$ and $V_4V_5$.

Because the two horizontal segments are in the middle of the top and bottom
faces of the box, rotations about them and reflections in the lateral faces of
the box extend $S$ to a TPMS $\tilde \Sigma$.  More specifically, $\tilde
\Sigma$ is invariant under the lattice $\Lambda$ spanned by $(2A,0,2)$,
$(-2A,0,2)$ and $(0,4B,0)$.  In the 3-torus $\R^3 / \Lambda$, $\Sigma = \tilde
\Sigma / \Lambda$ is a compact surface of genus $3$.  In
Figure~\ref{fig:fundpieces} (right) we show part of $\tilde \Sigma$ consisting
of eight copies of $S$.


\begin{remark}
	For crystallographers, the orthorhombic lattice spanned by $(4A,0,0)$,
	$(0,4B,0)$ and $(0,0,4)$ is probably more convenient.  This is responsible
	for the letter ``o'' in our naming.  The part shown in
	Figure~\ref{fig:fundpieces} (right) is actually a translational fundamental
	domain of this orthorhombic lattice.  The quotient of $\tilde\Sigma$ by this
	lattice is a double cover of $\Sigma$, hence of genus $5$. 
\end{remark}

\begin{figure}[hbt]
	\includegraphics[width=0.48\textwidth]{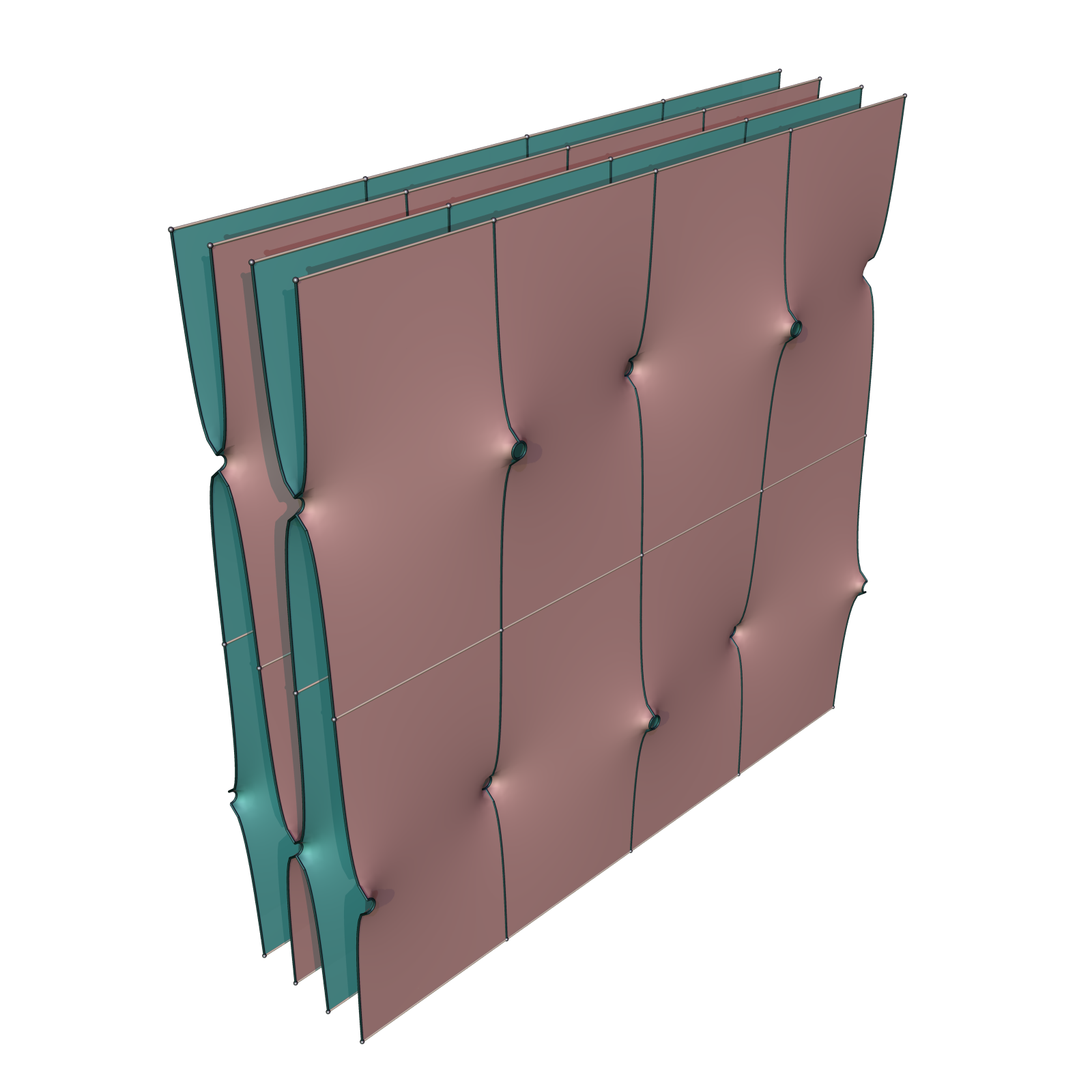}
	\includegraphics[width=0.48\textwidth]{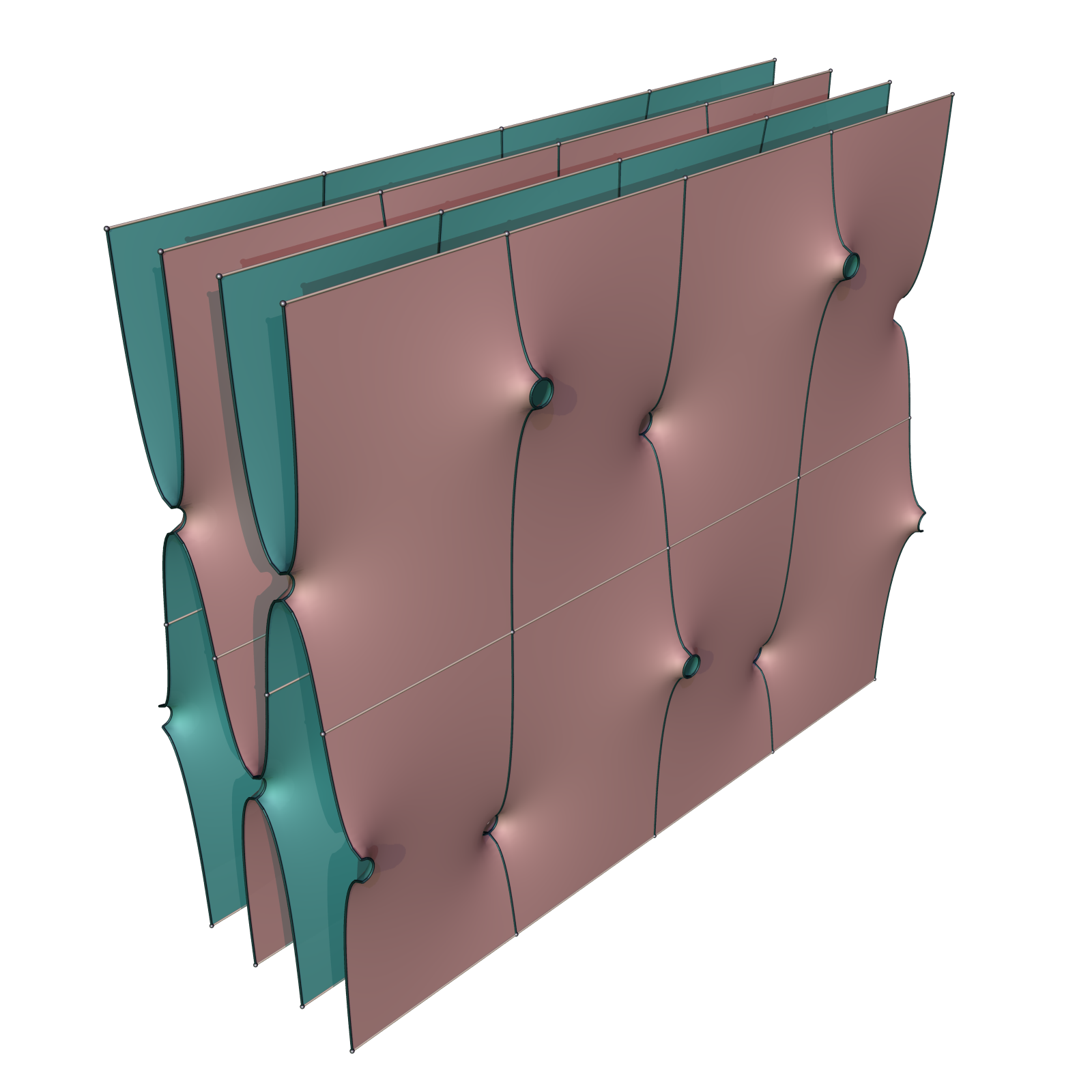}
	\caption{Surfaces in $\oP$ and $\H$ near catenoidal limits}
	\label{fig:nodes}
\end{figure}

We use $\cO$ to denote the set of all TPMS obtained in this way.  Two classical
families of surfaces in $\cO$ were already known to Schwarz~\cite{schwarz1890}.

Surfaces in the first family have an additional reflectional symmetry in the
plane $z=0$.  Then, because of the inversional symmetry in the origin, these
surfaces must also contain the $z$-axis, which serves as the axis of an order-2
rotational symmetry.  This $2$-parameter family contains Schwarz' $P$ surface
and belongs to the Meeks' family $\cM$~\cite{meeks1990}.  It is known as $\oP
b$ in the literature to distinguish from another orthorhombic deformation
family $\oP a$; see~\cite{fischer1989,fogden1992}.  In this paper, we simply
write $\oP$ in place of $\oP b$.  An example of $\oP$ with small $B$ is shown
in Figure \ref{fig:nodes} (left).

Surfaces in the second family have an additional order-3 rotational symmetry
about a line in the $y$-direction.  The rotational axis necessarily passes
through an end of fixed boundary.  This $1$-parameter family is Schwarz' $\H$
family.  An example, again with small $B$, is shown in Figure \ref{fig:nodes}
(right).

\medskip

The main purpose of this paper is to establish the existence of a new
2-parameter family described in the following theorem, and study its
properties.

\begin{theorem}\label{thm:main}
	There exists a 2-parameter continuous family $\oH$ in $\cO$ that contains
	Schwarz' $\H$ surfaces as a subfamily.  Surfaces in $\oH$ do not belong to
	the Meeks family.  That is, the branched values of the Gauss map do not form
	four antipodal pairs.  In fact, the only Meeks surfaces in $\cO$ are the
	$\oP$ surfaces.  However, the closure of $\oH$ intersects $\oP$ in a
	$1$-parameter family of TPMS.
\end{theorem}

The intersection $\overline\oH \cap \oP$ will be explicitly described in terms
of elliptic integrals.

\medskip

We now provide the motivation that leads to the discovery of $\oH$.

By opening nodes among 2-tori, Traizet~\cite{traizet2008} constructed TPMS that
looks like horizontal planes connected by catenoidal necks.  In the degenerate
limit, the catenoidal necks become nodes whose positions have to satisfy a
balancing equation, formulated in terms of elliptic functions, and a
non-degeneracy condition.

For surfaces of genus 3, one needs to open two nodes between two tori.  In the
limit, it degenerates to a two-sheeted torus with two singular points.  Let
$(T_1, T_2) \in \C^2$ be vectors that span the limit torus, and write $T_3 =
-T_1-T_2$.  Assume that limit positions of the two nodes are $p_1$ and $p_2$,
respectively.  Up to a translation, we may assume that $p_1 = 0$.  Write $p_2 =
x T_1 + y T_2$ with $(x,y) \in [0,1]^2$.  Then $p_1, p_2$ form a balance
configuration if (see \cite[\S 4.3.3]{traizet2008})
\begin{equation}\label{eq:balanced}
	\zeta(p_2) = x \eta_1 + y \eta_2 \ ,
\end{equation}
where $\zeta$ is Weierstrass Zeta function, which is quasi-periodic in the
sense that $\zeta(z+T_i)-\zeta(z) = \eta_i = 2 \zeta(T_i/2)$, $i=1,2,3$.
Traizet proved that, if such a balanced configuration is non-degenerate, then
there exists a family of triply periodic minimal surfaces that limits in this
configuration.

\begin{remark}
	If we follow~\cite{traizet2008} more closely, we would need to consider an
	infinite sequence of tori indexed by $\Z$, and open a node between each
	adjacent pair under the periodicity assumption that $p_{k+2} = p_k + T$ for
	some $T \in \C$.  The current paper only deals with the monoclinic special
	case with $T=0$, hence the simplified formulation.
\end{remark}

Recall that $\eta_1+\eta_2+\eta_3=0$.  Hence for any $T_1$ and $T_2$, there are
three trivial balanced configurations, given by $(x,y)=(1/2,1/2)$, $(0,1/2)$
and $(1/2,0)$, respectively.  In other words, if $p_2$ is a 2-division point,
the balance equation is automatically solved for any torus.  In particular,
when $|T_1|=|T_2|$, the configuration $x=y=1/2$ is the Traizet limit of the
$\oP$ ($\oP b$) family; when $T_2/T_1$ is purely imaginary, the configuration
$x=y=1/2$ is the Traizet limit of the $\oP a$ family, and the configuration
$(x,y)=(0,1/2)$ (or $(1/2,0)$) is the Traizet limit of an orthorhombic
deformation family of CLP surfaces (termed oCLP' in \cite{fogden1992}).
Another special case is the rhombic $60^\circ$ torus.  The hexagonal symmetry
implies two non-trivial balanced configurations with $x=y=1/3$ and $x=y=2/3$,
which is the Traizet limit of the $\H$ family (non-Meeks).  Traizet limit of
these classical TPMS are illustrated in Figure~\ref{fig:classic}.

\begin{figure}[hbt]
	\includegraphics[width=0.6\textwidth]{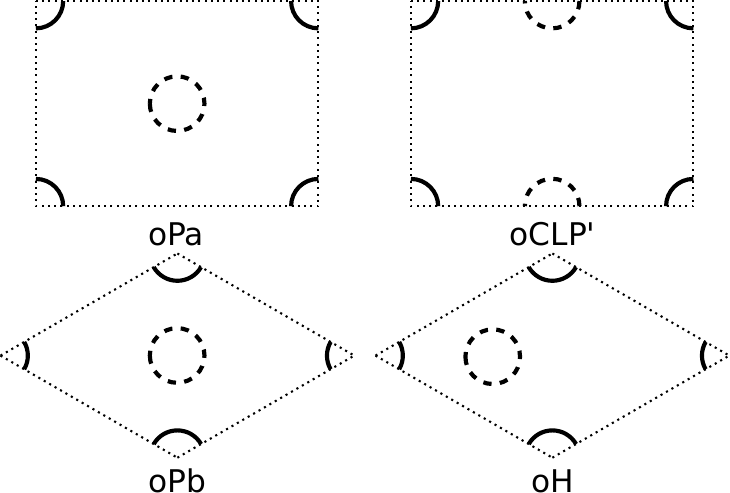}
	\caption{
		Balanced configurations that give Traizet limits of classical families of
		TPMS.  The solid circle is centered at $p_1$ and the dashed circle is
		centered at $p_2$.
	} \label{fig:classic}
\end{figure}

More generally, we consider on rhombic tori spanned by $T_{1,2}=\exp(\pm i
\theta/2)$ the balanced configurations with $x=y$.  The symmetry suggests that
these are Traizet limit of $\cO$ surfaces with $B \to 0$.  To see this, just
rotate the configuration to place $T_3$ in the $z$-direction and open nodes in
the $y$-direction.  The balance equation for such a configuration is given by 
\begin{equation}\label{eq:traizet1}
	x \zeta(T_3/2) = \zeta(x T_3)/2 
\end{equation}
Our choice of conjugate vectors $T_1, T_2$ guarantees real values on both
sides.  The solution set to this balanced equation is shown in
Figure~\ref{fig:diagsols}.  The vertical line $x=1/2$ is the trivial locus,
giving the Traizet limit of the $\oP$ surfaces.

\begin{figure}[hbt]
	\includegraphics[width=0.4\textwidth]{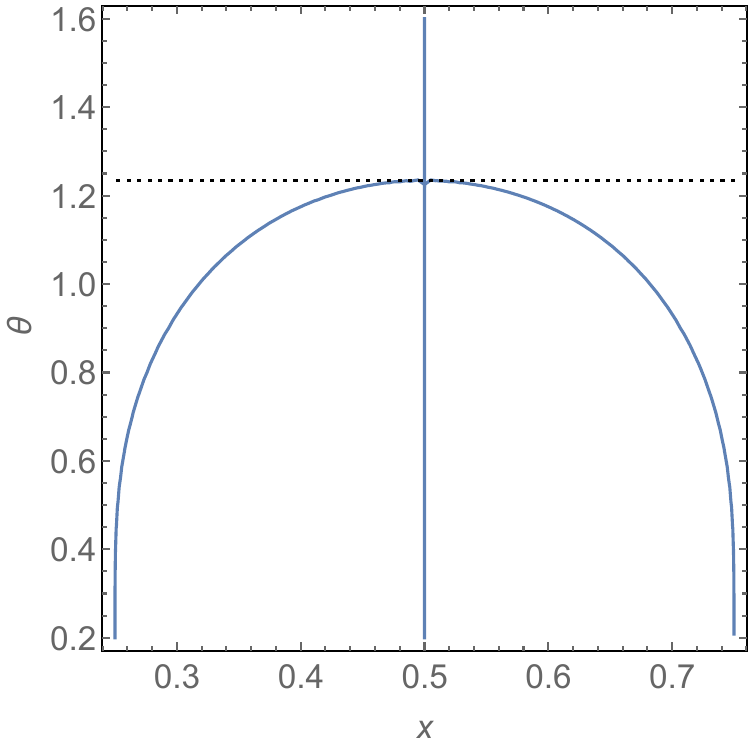}
	\caption{
		Solution set $(x,\theta)$ to the balance equation~\eqref{eq:traizet1} on the
		diagonal of rhombic tori.  The dotted horizontal line mark the position of
		$\theta^*$.
	}
	\label{fig:diagsols}
\end{figure}

But we also see a second, non-trivial locus, which is the motivation of the
current project.  This ``exotic'' locus has been noticed independently by both
authors, and probably by many other people in the minimal surface community.
We will see that balanced configurations on this locus are all non-degenerate,
so they are indeed Traizet limits.  In fact, they are the Traizet limits of the
$\oH$ family.  Our discovery of $\oH$ is actually the result of an attempt to
push TPMS away from these Traizet limits.  

\medskip

The two loci of Traizet limits intersect at $x=1/2$ and $\theta=\theta^*\approx
1.23409 \approx 70.7083^\circ$.  The balanced configuration at the intersection
is degenerate.

The torus $\T^*$ at the intersection is of particular significance.  It is the
only rhombic torus on which there exists a meromorphic 1-form with double order
pole at $0$ and double order $0$ at $T_3/2$ and only real periods.  This was
exploited for the construction of translation invariant helicoids with
handles~\cite{hoffman1999,weber2009}.  This meromorphic 1-form can be
constructed geometrically as follows (see Figure \ref{fig:slit70}): Take the
complex plane, and slit it along the interval $[-1,1]$ on the real axis.  Then
identify the top (resp.\ bottom) edge of $[-1,0]$ with the bottom (resp.\ top)
edge of $[0,1]$.  The result is a torus carrying a cone metric with two cone
points, of cone angle $6\pi$ at the point identified with $\{-1,0,1\}$, and of
cone angle $-2\pi$ at $\infty$.  The corresponding 1-form has thus a double
order pole at $0$, and a doubly order zero at $\infty$.  Its periods are
obviously real, and the symmetry of the slit ensures that the torus is rhombic.
The same torus with flat metric is nothing but $\T^*$.

We will revisit the Traizet limit in Section \ref{sec:traizet} in the framework
of our parametrization of $\oH$.  We first prove that the non-trivial locus is
non-degenerate, and unique in the sense that for every $0<\theta<\theta^*$,
\eqref{eq:traizet1} has a unique solution $0 < x < 1/2$.  Then
\eqref{eq:traizet1} will be reformulated in terms of elliptic integrals,
leading to an explicit formula for the non-trivial locus.  We will also recover
$\theta^*$, not only as the end point of the Traizet limit of $\oH$, but also
as the Traizet limit of the intersection $\overline\oH \cap \oP$.

\begin{figure}[hbt]
	\includegraphics[width=0.4\textwidth]{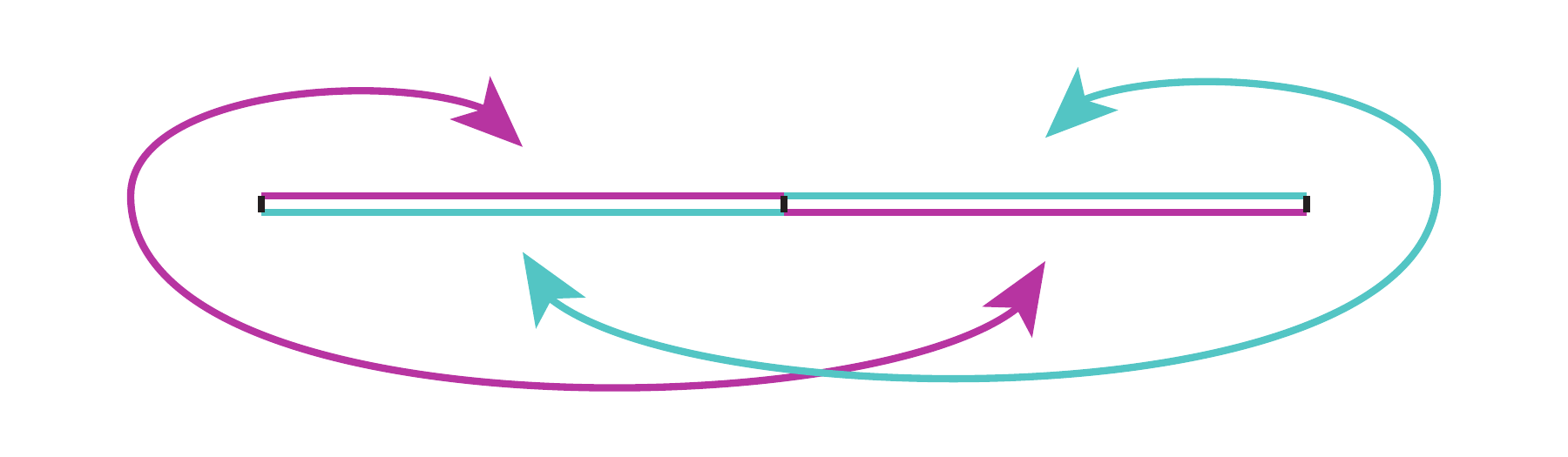}
	\caption{Model for the torus $\T^*$ }
	\label{fig:slit70}
\end{figure}

\medskip

Our paper is organized as follows: 

In Section \ref{sec:period}, we describe the Weierstrass data for surfaces in
$\cO$, prove their embeddedness, and formulate the period problem, depending on
three real positive parameters $a, b$ and $t$. The case $a=b$ corresponds to
the $\oP$ surfaces, where the period problem is automatically solved.  In the
case $a\ne b$, the period problem becomes 1-dimensional.

In Section \ref{sec:gauss} we show that, if $a \ne b$, the branched values of
the Gauss map can \emph{not} be antipodal.  This proves that $\cO \cap \cM =
\oP$, and that any solution with $a \ne b$ (namely $\oH$) lies outside $\cM$.

Section \ref{sec:exist} is dedicated to the existence proof of $\oH$.  We show
that for any choice of $a \ne b$, there is a value of $t$ that solves the
period problem.  This is accomplished through a careful asymptotic analysis of
the period integrals.  We also conjecture the uniqueness of $t$ based on
numerical experiments.

To prove that $\oP$ and the closure of $\oH$ have a non-empty intersection, we
consider in Section \ref{sec:elliptic} a modified period problem that
eliminates the trivial solutions coming from $\oP$. It turns out that the
intersection can be explicitly described in terms of elliptic integrals.

In section \ref{sec:traizet} we study the Traizet limit of $\oH$.  In
particular, the loci of~\eqref{eq:traizet1} will receive another explicit
description in terms of elliptic integrals, and the intersection of the loci
will be recovered in two different ways.  We also locate the Traizet limit of
$\H$ family on the locus.  It is then possible to find a continuous deformation
path within the space of TPMS of genus three, starting from an $\H$ surface and
ending with an $\oP$ surface, that passes through a sufficiently small
neighborhood of the Traizet limit.

\medskip

Despite different appearances, motivations and focus points, our
parametrization of $\oH$, as well as many computations, share similarities with
our previous work on $\oDD$~\cite{chenweber1}.  So we will, whenever
appropriate, refer the readers to~\cite{chenweber1} for details.  We also omit
technical details in Sections~\ref{sec:elliptic} and~\ref{sec:traizet}, where
integral tables in~\cite{byrd1971} are used for the computations involving
elliptic integrals.

\subsection*{Acknowledgements}

We are grateful to the anonymous referee for suggestions and corrections after
carefully reading a previous version of the manuscript.

\section{Weierstrass Data and the Period Problem}
\label{sec:period}

We parameterise a surface in $\cO$ using a Weierstrass representation defined
on the upper half plane such that the real axis is mapped to the boundary of
the octagon $S$.  Let the vertices of $S$ be labeled by $V_1, V_2, \cdots, V_8$
as in Figure~\ref{fig:fundpieces} (left).  Denote the preimage of $V_k$ by $v_k
\in \R$, and assume that $v_1<v_2<\ldots<v_8$. 

\begin{figure}[hbt]
	\includegraphics[width=0.8\textwidth]{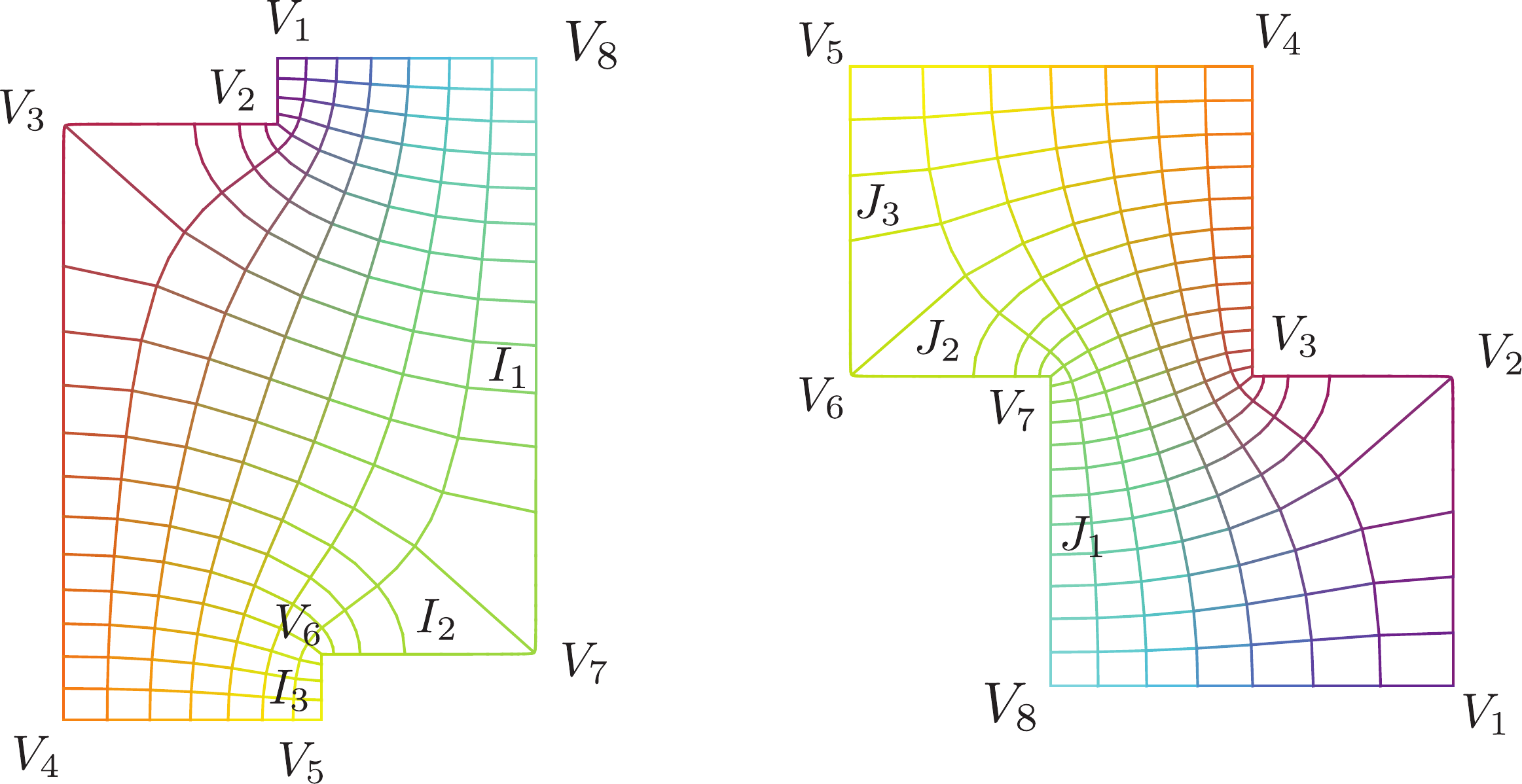}
	\caption{Images of a fundamental piece under $\Phi_1$ and $\Phi_2$.}
	\label{fig:phidomains}
\end{figure}

Given an $\cO$ surface, denote by $dh$ its height differential and by $G$ its
Gauss map. Let $\phi_1:=dh \cdot G$ and $\phi_2:=dh / G$.  The assumed boundary
symmetries of the surface imply that $\Phi_j: z \mapsto \int^z \phi_j$ ($j=1$
or $2$) map the upper half plane to ``right angled'' Euclidean octagons.  The
interior angle is $270^\circ$ at $\Phi_1(v_2)$, $\Phi_1(v_6)$ and
$\Phi_2(v_3)$, $\Phi_2(v_7)$.  Indeed, the Gauss map is vertical at $V_2$,
$V_3$, $V_6$ and $V_7$, hence these vertices are the poles and the zeros of
$G$.  Interior angles at all other vertices are $90^\circ$; see Figure
\ref{fig:phidomains}.

We may assume that the inversion is represented by the transform $\iota:z \mapsto
-1/z$, hence the inversion center of the minimal octagon at the origin is
represented by $i$ in the upper half plane.  Then we assume the eight points
$v_i$ to be $-t < -a < -1/b < -1/t < 1/t < 1/a < b < t$ for $t>1$.

Such maps are given by Schwarz-Christoffel maps.  More specifically, we have
\begin{align*} 
	\phi_1 :=  - \rho \, & \frac{ (z+a)^{+1/2} (z+1/b)^{-1/2}  (z-1/a)^{+1/2}(z-b)^{-1/2}}
	{(z+t)^{1/2} (z+1/t)^{1/2}(z-1/t)^{1/2}(z-t)^{1/2}}\, dz,\\
	\phi_2 := \frac{1}{\rho} & \frac{ (z+a)^{-1/2} (z+1/b)^{+1/2} (z-1/a)^{-1/2}(z-b)^{+1/2}}
	{(z+t)^{1/2} (z+1/t)^{1/2}(z-1/t)^{1/2}(z-t)^{1/2}}\, dz,\\
	dh := & \frac{i} {(z+t)^{1/2} (z+1/t)^{1/2}(z-1/t)^{1/2}(z-t)^{1/2}}\, dz \ .
\end{align*}
Here, the real positive Lop{\'e}z-Ros factor $\rho$ determines the scaling of the
image domains.  The Gauss map is given by
\[
	G(z) = i \rho (z-1/a)^{+1/2}(z+a)^{+1/2} (z+1/b)^{-1/2}(z-b)^{-1/2}\ .
\]

\begin{proposition}\label{prop:almost}
	Up to congruence and dilation, the image of the upper half plane under the
	map
	\begin{equation} \label{eq:weierstrass}
		F: z \mapsto \re \int^z (\omega_1, \omega_2, \omega_3) = \re \int^z \left(
		\frac{1}{2}(\phi_2-\phi_1), \frac{i}{2}(\phi_2+\phi_1), dh \right)
	\end{equation}
	is \emph{almost} the fundamental octagon of an $\cO$ surface in the following
	sense:
	\begin{itemize}
		\item The intervals $v_8v_1$ and $v_4v_5$ are mapped to straight segments
			in the $x$-direction, but not necessarily in the $y=0$ plane.

		\item The other intervals are mapped to planar symmetry curves in vertical
			planes.  More specifically:
			\begin{itemize}
				\item the interval $v_1v_2$ (resp.\ $v_5v_6$) is mapped into the plane $x=+A$ (resp.\ $-A$);

				\item the interval $v_2v_3$ (resp.\ $v_6v_7$) is mapped into the plane $y=-B$ (resp.\ $+B$);

				\item the interval $v_3v_4$ (resp.\ $v_7v_8$) is mapped into the plane $x=+A'$ (resp.\ $-A'$).
			\end{itemize}

		\item The image is symmetric under the inversion $\iota$ in the image of $i$.
	\end{itemize}
\end{proposition}

The proof is a straightforward modification of the proof in \cite{chenweber1}.

\begin{proposition}
	All minimal octagons constructed by Proposition \ref{prop:almost} are
	embedded.  In particular, the triply periodic minimal surfaces in $\cO$,
	generated by extending the octagon across symmetry lines, are embedded as
	well.
\end{proposition}

\begin{proof}
	We note that the space $\cS$ of these octagons is parametrized by the
	parameters $a$, $b$, $t$ and $\rho$, and hence connected. Moreover, $\cS$
	contains known embedded surfaces, namely the $\oP$ and $\H$ surfaces of H.A.
	Schwarz. 

	Now let $\Sigma_1$ be an arbitrary surface in $\cS$, and connect it
	continuously to an embedded surface $\Sigma_0$ in $\cS$. Let $\Sigma$ be the
	first surface in this family that is not embedded anymore. We denote by
	$\partial \Sigma =F(\R\cup\{\infty\})$ the boundary of $\Sigma$, while
	$\Sigma = F(\R\times \R^+)$ denotes the interior.  By the maximum principle,
	the immersion fails to be embedded at a boundary point of $\Sigma$.

	The boundary $\partial \Sigma$ itself is parametrized by $F$ injectively: On
	the vertical segments between $V_5$ and $V_8$ and between $V_1$ and $V_4$,
	this follows because the height differential is real with constant sign. The
	straight horizontal segments are parametrized injectively because  the
	parametrization is conformal and non-degenerate.

	Thus it remains to show that the $\partial \Sigma$ cannot meet a point of
	$\Sigma$.  This is clear for most boundary segments $V_iV_{i+1}$ by the
	convex hull property of minimal surface.  But this argument could fail when
	the arc $V_7V_8$ is not coplanar with the arc $V_5V_6$, i.e.\ when $A \ne
	A'$.  If $A'<A$, for instance, then there could be an interior point $p_0$ of
	$\Sigma$ that meets $V_7V_8$.  If this is the case, $\Sigma$ would be
	tangential to $\partial \Sigma$ at $p_0$.

	By the explicit expression of the Gauss map, no point of $\Sigma$ has a
	normal vector parallel to the $xz$-plane, hence the tangent plane cannot be
	parallel to the $y$-direction.  It is therefore possible to find a curve
	$\gamma(t)$ on $\Sigma$, starting from $p_0$, whose tangent vector is in the
	direction of $(-1,y(t),1)$.  This curve can only terminate at a boundary
	point of $\partial \Sigma$ or at an interior point with a tangent plane
	containing the $y$-direction. The latter doesn't exist, and there is no
	component of the boundary having points with smaller $x$-coordinate and
	greater $z$-coordinate than $V_7$. This contradiction shows that $\bar\Sigma
	=\Sigma\cup \partial\Sigma$ is embedded.
\end{proof}

\medskip

For such a minimal octagon to lie in $\cO$, we must have have $A=A'$ so that
the curves $V_1V_2$ and $V_3V_4$ are coplanar, hence the image
of~\eqref{eq:weierstrass} is contained in an axis parallel box centered at the
origin.  Moreover, $V_8V_1$ and $V_4V_5$ must lie in the middle of,
respectively, the top and bottom faces of the box.  We now express these
conditions in terms of the periods of $\phi_1$ and $\phi_2$.   To this end, we
introduce notations for the edge lengths of the Euclidean octagons 
\[
	I_k := \left| \int_{v_k}^{v_{k+1}} \phi_1 \right|, \qquad
	J_k := \left| \int_{v_k}^{v_{k+1}} \phi_2 \right|
\]
for $1 \le k \le 7$.  These are positive real numbers that depend analytically
on the parameters $a, b, t$ and $\rho$. Note that by the inversional symmetry,
we have
\begin{equation}\label{eq:symmetry}
	I_k=I_{k+4} \quad\text{and}\quad J_k=J_{k+4}
\end{equation}
for $1 \le k \le 3$.

\begin{proposition}
	The image of the upper half plane under the Weierstrass
	representation~\eqref{eq:weierstrass} is the fundamental octagon of a surface
	in $\cO$ if and only if the following period conditions are satisfied:
	\begin{equation}
		\begin{aligned} \label{eq:periodcond}
			I_1+I_3 &= J_1+J_3 \ , \\
			I_2&= J_2 \ .
		\end{aligned}
	\end{equation}
\end{proposition}

\begin{proof}
	The curves $V_1V_2$ and $V_3V_4$ are coplanar if and only if
	\[
		\re \int_{v_2}^{v_3} \omega_1 = 0 \ .
	\]
	This is equivalent to
	\[
		\re \int_{v_2}^{v_3} (\phi_1 - \phi_2) = 0.
	\]
	Observe that on $v_2v_3$, the integrands of $\phi_1$ and $\phi_2$ are both
	negative real.  So the equation above can be written as $I_2 = J_2$, which is
	the second period condition.

	The top segment $V_8V_1$ lies in the middle of the top face if and only if
	\[
		\re \int_{v_1}^{v_2} \omega_2 = \re \int_{v_7}^{v_8} \omega_2 \ .
	\]
	This is equivalent to 
	\[
		\im \int_{v_1}^{v_2} (\phi_1+\phi_2) = \im \int_{v_7}^{v_8} (\phi_1+\phi_2) \ .
	\]
	Observe on $v_1v_2$ that the integrand in $\phi_1$ (resp.\ $\phi_2$) is
	negative (resp.\ positive) imaginary, and on $v_7v_8$ that the integrand in
	$\phi_1$ (resp.\ $\phi_2$) is positive (resp.\ negative) imaginary. So the
	equation above can be written as
	\[
		I_1 - J_1 = J_7 -I_7 = J_3 - I_3,
	\]
	where the second equation follows from the symmetry~\eqref{eq:symmetry}.
	This proves the first period condition.

	If the period conditions are satisfied, then by the inversional symmetry $\iota$, the
	curves $V_5V_6$ and $V_7V_8$ must also be coplanar, and the segment $V_4V_5$
	must also lie in the middle of the bottom box. 
\end{proof}

We can eliminate $\rho$ by taking the quotient of the two equations, therefore:

\begin{corollary} \label{cor:periodcondition}
	If 
	\[
		Q_I := \frac{I_1+I_3}{I_2} = \frac{J_1+J_3}{J_2} =: Q_J
	\]
	or, equivalently, if
	\begin{equation}\label{eq:quotient}
		Q := Q_I - Q_J = \frac{I_1+I_3}{I_2} - \frac{J_1+J_3}{J_2} = 0
	\end{equation}
	for some choice of $a,b,t$, then $\rho \in \R_{>0}$ can be uniquely adjusted
	so that the period conditions~\eqref{eq:periodcond} are satisfied.
\end{corollary}

Thus we have expressed the period condition as a single equation $Q=0$, where
$Q$ depends on three parameters $a,b,t$.  

Note that when $a=b$, we have additional symmetries:

The involution $\sigma_1: z \mapsto
-\overline{z}$ transforms the Weierstrass data by
\[
	\sigma_1^*dh = \overline{dh(z)} \quad \text{and} \quad
	G(\sigma_1(z)) \overline{G(z)} = \rho^2.
\]
Consequently, we have $I_k = \rho^2 J_{4-k}$ for $1 \le k \le 3$, so
the period conditions~\eqref{eq:periodcond} are satisfied automatically with
$\rho=1$.  In this case, it can be explicitly verified that the positive
imaginary axis is mapped by the Weierstrass
representation~\eqref{eq:weierstrass} to the vertical straight segment between
the middle points of $V_4V_5$ and $V_8V_1$, and $\sigma_1$
induces an order-2 orientation-reversing rotation of the surface
around this segment.  

On the other hand, the involution $\sigma_2: z \mapsto 1/{\bar z}$ induces an
order-2 reflection in the $z=0$ plane.  This can be seen by observing that
$\sigma_1\circ \sigma_2=\iota$.

\medskip

To simplify our computations in the following sections, we employ the
substitution $\zeta=z-1/z$, which is monotone on the positive real axis. We
also replace $a-1/a$ by $\alpha$, $b-1/b$ by $\beta$, and $t-1/t$ by $\tau$ so
that $-\tau < -\alpha < \beta < \tau$. Then the 1-forms $\phi_1$ and $\phi_2$
become
\begin{align*} 
	\phi_1 = {}& -\rho (\zeta + \alpha)^{+1/2} (\zeta - \beta)^{-1/2} (\zeta^2 - \tau^2)^{-1/2} (\zeta^2 + 4)^{-1/2}\, d\zeta \ ,\\
	\phi_2 = {}& \frac{1}{\rho} (\zeta +\alpha)^{-1/2} (\zeta -\beta)^{+1/2} (\zeta^2 - \tau^2)^{-1/2} (\zeta^2 + 4)^{-1/2}\, d\zeta \ ,
\end{align*}
and the Gauss map is simplified to
\begin{equation}\label{eq:altGauss}
	G(\zeta)=\rho i (\zeta + \alpha)^{+1/2} (\zeta-\beta)^{-1/2}.
\end{equation}

In the rest of the paper, the original parametrization is understood whenever
Latin letters $a,b,t,z$ are used, and the simplified parametrization is
understood whenever Greek letters $\alpha,\beta,\tau,\zeta$ are used.  This
should not cause any confusion.

\section{Branched Values of the Gauss Map}
\label{sec:gauss}

In this section, we will show that the branched values of the Gauss map are
never antipodal with $a \ne b$.  As a consequence, the only surfaces in $\cO$
that belong to the Meeks family $\cM$ are the surfaces in $\oP$. The arguments
don't require the period condition to be satisfied and are purely algebraic.

\begin{theorem}
	The branched values of the Gauss map of a surface in $\cO$ are antipodal if
	and only if $a=b$.
\end{theorem}

\begin{proof}
	We begin by locating the branched points of the Gauss map in the fundamental
	octagon. By a result of Meeks \cite{meeks1990}, the branched points of a TPMS
	of genus 3 are precisely the inversion centers of the surface. They are
	situated, in the fundamental octagon, at the center of the octagon and at the
	end points of the fixed boundary segments.

	The octagon center corresponds to $i$ in the upper half plane, so that $G(i)$
	is a branched value. Three more branched points and values are obtained after
	extending the octagon by reflections.  We then have four branched values,
	namely $\pm G(i)$ and $\pm \overline{G(i)}$. Their stereographic images on
	the 2-sphere lie at the vertices of a horizontal rectangle, symmetric in the
	planes $x=0$ and $y=0$. 

	The end points of the fixed boundary segments correspond to $\pm t$ and $\pm
	1/t$ in the parameter domain.  Because of the inversional symmetry, they
	provide only two branched values $G(t)$ and $G(-t)$.  These both lie on the
	positive imaginary axis, and their stereographic images on the 2-sphere lie
	on the upper half-circle with $y>0$ and $x=0$.  Extending the octagon by
	reflections gives two more branched values at $-G(t)$ and $-G(-t)$, whose
	stereographic images lie on the lower half-circle.  The stereographic images
	of these four branched values then form a quadrilateral in the plane $x=0$
	symmetric to the plane $y=0$.

	We show an example for the location of the eight branched values in Figure
	\ref{fig:gaussbranched}.

	\begin{figure}[hbt]
		\includegraphics[width=0.4\textwidth]{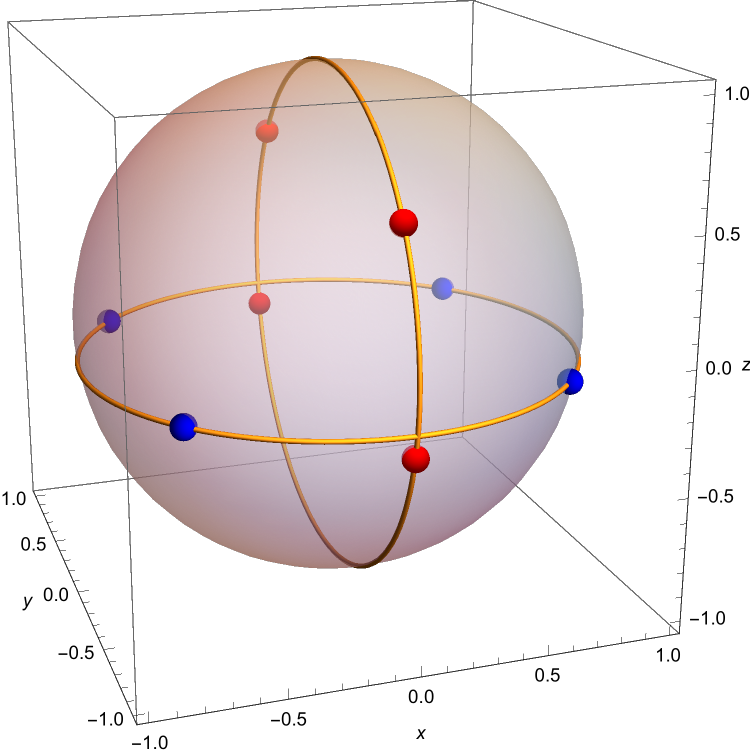}
		\caption{Branched values of the Gauss map}
		\label{fig:gaussbranched}
	\end{figure}

	In order that these eight branched values occur in antipodal pairs, the first
	quadrilateral must lie in the plane $z=0$, while the second quadrilateral
	must be a rectangle.  These conditions mean, in terms of the Gauss map, that
	$|G(i)|=1$ and $G(t)G(-t)=-1$.  We then compute from~\eqref{eq:altGauss} that
	\[
		\rho^2 \sqrt\frac{\alpha^2+4}{\beta^2+4} = 1\quad\text{and}\quad
		\rho^2 \sqrt\frac{\tau^2-\alpha^2}{\tau^2-\beta^2} = 1,
	\]
	which forces $\alpha = \pm\beta$, hence $a=b$ under the constraint $1/t < 1/a
	< b < t$.
\end{proof}

The reader might be curious about the parameter values for the Schwarz $\H$
surfaces within this representation.  These are difficult to determine
explicitly.  But we know that, among the eight branched values of the Gauss map
of an $\H$ surface, there is one and only one antipodal pair.  This implies
either $G(t)=G(-1/t)=i$ or $G(-t)=G(1/t)=i$.  We then obtain
from~\eqref{eq:altGauss} a necessary condition of the parameters for $\H$,
namely
\[
	\rho^2 \frac{\tau + \alpha}{\tau - \beta} = 1 \quad\text{or}\quad \rho^2 \frac{\tau - \alpha}{\tau + \beta} = 1.
\]
In view of Conjecture~\ref{conj:unique} in the next section, we believe that
this condition is also sufficient.

We see from~\eqref{eq:altGauss} that, for any reals $\alpha,\beta,\rho$ with
$-\alpha < \beta$ and $\rho > 0$, there is a unique $\zeta^* \in
(\infty, -\alpha) \cup (\beta, \infty) \cup \infty$ such that $G(\zeta^*) = i$.
In other words, there must be a point on the boundary of the octagon, namely
the image of $\zeta^*$ under~\eqref{eq:weierstrass}, where the normal vector
points in the $y$ direction.  For the $\oP$ family, $\rho = 1$ and $\alpha =
\beta$, hence $\zeta^* = \infty$.  For the $\H$ family, our calculation above
shows that $\zeta^* = \pm \tau$.  So $\cO$ is divided in two parts, depending
on the image of $\zeta^*$ being on the fixed boundary (as $\oP$) or on the free
boundary of the octagon; Schwarz $\H$ family lies on the interface.

\begin{remark}
	Using the order-$3$ rotational symmetry of the $\H$ surfaces, a computer
	algebra system gives the explicit expressions
	\begin{align*}
		\alpha &= \tau \frac{4(\tau^2+4)\sqrt{\tau^4-56\tau^2+16} - (3\tau^4-40\tau^2+48)}{7\tau^4+88\tau^2-16} \\
		\beta &= \tau \frac{4(\tau^2+4)\sqrt{\tau^4-56\tau^2+16} + (3\tau^4-40\tau^2+48)}{7\tau^4+88\tau^2-16}
	\end{align*}
	for the parameters $\alpha$ and $\beta$ in terms of $\tau$.  Then the period
	problem seems automatically solved, at least numerically.
\end{remark}

\section{Existence of Non-Trivial Solutions}
\label{sec:exist}

Recall that $0 < 1/t < 1/a < b < t$, and the periodic
condition~\eqref{eq:quotient} as we copy below
\[
	Q(a,b;t) = \frac{I_1+I_3}{I_2} - \frac{J_1+J_3}{J_2} = 0.
\]
The quantity $Q$ is our focus in the remaining sections of this paper.  From
now on, we will ignore the Lop\'ez-Ros factor $\rho$ in our calculations, since
$Q$ is independent of this factor.

We now prove the main theorem of this paper.
\begin{theorem} \label{thm:exist}
	If $a = b$, the period condition~\eqref{eq:quotient} is solved for any choice
	of $t$.

	If $a < b$, then there exists a value of $t$ that solves the period
	condition~\eqref{eq:quotient}.
\end{theorem} 

The case $a = b$ has been discussed in Section~\ref{sec:period}.  The case $a <
b$, as well as the existence of $\oH$, follows from the continuity of $Q$ in
$t$, and the following proposition.

\begin{proposition} \label{prop:lim1}
	If $1/t<1/a<b<t$ and $a<b$ then
	\begin{align}
		\lim_{t \to b+} Q(a,b;t) &< 0 \ ,\label{eq:ttob}\\
		\lim_{t \to +\infty} Q(a,b;t) &= +\infty \ .\label{eq:ttoi}
	\end{align}
\end{proposition}

The remainder of this section is devoted to the proof of this proposition.

\medskip

\begin{proof}[Proof of~\eqref{eq:ttob}]
	The argument in~\cite{chenweber1} applies with slight modification.  As $t
	\to b+$, all periods have finite positive limits, with the exceptions
	$\lim_{t \to b+} J_3=0$ and $\lim_{t \to b+} I_2$ diverges to $+\infty$.
	Thus
	\[
		\lim_{t\to b+} \frac{I_1+I_3}{I_2} = 0 \quad\text{and}\quad \lim_{t\to b+} \frac{J_1+J_3}{J_2} >0 \ ,
	\]
	and \eqref{eq:ttob} follows.
\end{proof}

\begin{proof}[Proof of~\eqref{eq:ttoi}]
	The proof is similar to the argument in~\cite{chenweber1}.  Recall that the
	substitution $\zeta=z-1/z$ is monotonically increasing for $z>0$, and write
	$\alpha=a-1/a$, $\beta=b-1/b$, $\tau=t-1/t$ as before.

	For the periods in the denominators, we first note that
	\begin{align*} 
		\lim_{\tau \to \infty} \tau \cdot I_2(\alpha,\beta;\tau) &= \int_{-\alpha}^{\beta} \frac{1}{\sqrt{\zeta^2+4}} \sqrt\frac{\alpha+\zeta}{\beta-\zeta}\, d\zeta,\\
		\lim_{\tau \to \infty} \tau \cdot J_2(\alpha,\beta;\tau) &= \int_{-\alpha}^{\beta} \frac{1}{\sqrt{\zeta^2+4}} \sqrt\frac{\beta-\zeta}{\alpha+\zeta}\, d\zeta
	\end{align*}
	are all finite.  Their difference
	\begin{align*}
		\lim_{\tau \to \infty} \tau \cdot (I_2-J_2) &= \int_{-\alpha}^{\beta} \frac{2\zeta+\alpha-\beta}{\sqrt{(\zeta^2+4)(\alpha+\zeta)(\beta-\zeta)}}\, d\zeta\\
		&= \int_{-\gamma}^{\gamma} \frac{2\xi\, d\xi}{\sqrt{((\xi-\alpha/2+\beta/2)^2+4)(\gamma^2-\xi^2)}}\\
		&= \int_{0}^{\gamma} \frac{2\xi\, d\xi}{\sqrt{\gamma^2-\xi^2}}\Big(\frac{1}{\sqrt{(\xi-\alpha/2+\beta/2)^2+4}}-\frac{1}{\sqrt{(\xi+\alpha/2-\beta/2)^2+4}}\Big)\ ,
	\end{align*}
	where $\gamma = (\alpha+\beta)/2$ and $\xi = \zeta - (\beta - \alpha)/2$, is
	negative when $\alpha < \beta$.  Hence we have
	\begin{equation}\label{eq:denominator}
		\lim_{\tau \to \infty} \tau I_2 < \lim_{\tau \to \infty} \tau J_2
	\end{equation}
	for all $0<\alpha<\beta$.

	The periods in the numerators have logarithmic asymptotics.  For instance, as
	$\tau \to \infty$, 
	\begin{align*} 
		\tau \cdot J_3(\alpha,\beta;\tau)
		&= \int_\beta^\tau \frac{\tau}{\sqrt{\zeta^2+4}\sqrt{\tau^2-\zeta^2}} \sqrt\frac{\zeta-\beta}{\zeta+\alpha}\, dz \\
		&\sim \int_\beta^\tau \frac{\tau}{\zeta\sqrt{\tau^2-\zeta^2}} \, d\zeta\\
		&\sim \log t,
	\end{align*}
	hence $\tau \cdot I_3(\alpha,\beta;\tau)$ diverges to $+\infty$ as $\tau \to
	\infty$.  Fortunately, the integrals $I_1$ and $J_1$ (and $I_3$ and $J_3$)
	have the same logarithmic singularities.  By the dominated convergence
	theorem, we obtain the following limits:
	\begin{equation} \label{eq:numerator}
		\begin{aligned} 
			\lim_{\tau\to\infty} \tau \cdot (I_1-J_1)
			&= -\lim_{\tau\to\infty} \int_{-\tau}^{-\alpha} \frac{\tau(\alpha+\beta)}{\sqrt{(\tau^2-\zeta^2)(\zeta^2+4)(\beta-\zeta)(-\alpha-\zeta)}}\, d\zeta\\
			&= -\int_{-\infty}^{-\alpha} \frac{\alpha+\beta}{\sqrt{\zeta^2+4}\sqrt{\beta-\zeta}\sqrt{-\alpha-\zeta}}\, d\zeta,\\
			\lim_{\tau\to\infty} \tau \cdot (I_3-J_3)
			&= \lim_{\tau\to\infty} \int_{\beta}^{\tau} \frac{\tau(\alpha+\beta)}{\sqrt{\tau^2-\zeta^2}\sqrt{\zeta^2+4}\sqrt{\zeta-\beta}\sqrt{\zeta+\alpha}}\, d\zeta\\
			&= \int_{\beta}^\infty \frac{\alpha+\beta}{\sqrt{\zeta^2+4}\sqrt{\zeta-\beta}\sqrt{\zeta+\alpha}}\, d\zeta.
		\end{aligned}
	\end{equation}
	Note that both integrals are finite and non-zero.

	Finally, we write
	\[
		Q(\alpha,\beta;\tau) = \frac{\tau (I_1 - J_1) + \tau (I_3 - J_3)}{\tau I_2 } + \tau (J_1 + J_3) \Big[\frac{1}{\tau I_2 } - \frac{1}{\tau J_2} \Big].
	\]
	The part in the square bracket is positive by~\eqref{eq:denominator}.  As
	$\tau\to\infty$, the first fraction is finite by~\eqref{eq:numerator}, and
	$\tau (J_1+J_3) \to +\infty$.  This concludes the proof of the proposition.
\end{proof}

Before ending this section, we propose the following uniqueness conjecture
based on numeric experiments:
\begin{conjecture}\label{conj:unique}
	If $a < b$, then there exists a unique $t$ that solves the period
	condition~\eqref{eq:quotient}.
\end{conjecture} 

\section{Intersection with the Meeks-Locus}
\label{sec:elliptic}

We show in this section that $\overline\oH$ intersects $\oP$ in a 1-parameter
family.  To make this precise, we use on $\cO$ the topology induced by the
space of possible Weierstrass data, which are determined by the four real
parameters $a, b, t$ and $\rho$.  Clearly, the convergence of Weierstrass data
implies the locally uniform convergence of the minimal surfaces.

The goal is to explicitly determine the intersection of the Meeks locus 
\[
	\oP = \{(a,b,t): Q(a,b;t) = 0, a=b, 0<1/t<1/a<b<t\}
\]
with the closure of the non-Meeks locus
\[
	\oH = \{(a,b,t): Q(a,b;t) = 0, a \ne b, 0<1/t<1/a<b<t\}.
\]
Without loss of generality, we will focus on the case $a<b$ hence $\alpha <
\beta$.  The idea is to divide the function $Q(\alpha,\beta;\tau)$ by
$\beta-\alpha$ and take the limit for $\alpha \to \beta-$ to eliminate solutions
in the Meeks locus. We claim: 

\begin{theorem} \label{thm:elliptic}
	The intersection $\overline\oH \cap \oP$ is described by the equation
	\begin{equation} \label{eq:intersection}
		\KK(m_1) E(m_2) + m_2 \EE(m_1) K(m_2) = \KK(m_1) K(m_2),
	\end{equation}
	where
	\begin{align*} 
  	K(m) &= \int_0^{\pi/2} \frac{1}{\sqrt{1-m \sin^2(\theta)}}\, d\theta,  \\
  	E(m) &= \int_0^{\pi/2} {\sqrt{1-m \sin^2(\theta)}}\, d\theta
	\end{align*}
	are complete elliptic integrals of the first and the second kind,
	$\KK(m)=K(1-m)$ and $\EE(m)=E(1-m)$ are the associated elliptic integrals,
	using the moduli
	\[
  	m_1 =\frac{\alpha^2+4}{\tau^2+4}, \qquad
  	m_2 =\frac{\alpha^2}{\tau^2}\frac{\tau^2+4}{\alpha^2+4}.
	\]
	Note that $0<m_1,m_2<1$.
\end{theorem}

\begin{figure}[hbt]
	\includegraphics[width=0.4\textwidth]{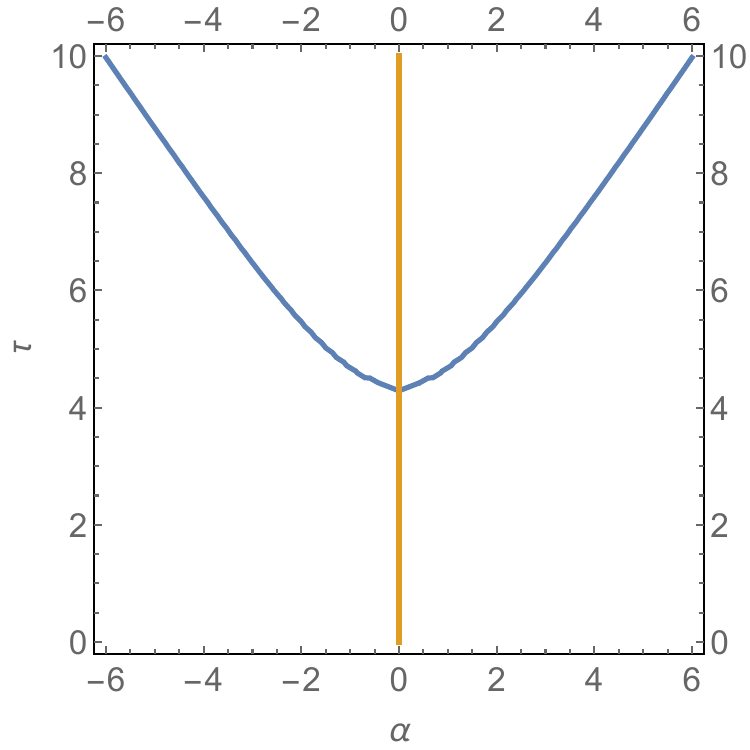}
	\caption{
		Solution set $(\alpha,\tau)$ to the period condition
		\eqref{eq:intersection} describing the intersection $\overline\oH \cap
		\oP$.
	}
	\label{fig:plotintersect}
\end{figure}

The theorem follows from the following proposition:

\begin{proposition}\label{prop:tildeQ}
	The function
	\[
		\tilde Q(\alpha,\beta;\tau) = \frac{Q(\alpha,\beta;\tau)}{\beta-\alpha}
	\]
	extends analytically to $\alpha=\beta$ by
	\[
		\tilde Q(\alpha,\alpha;\tau) = \frac{\tau^2}{\alpha^2}\frac{\tau}{\tau^2-\alpha^2}\sqrt{\frac{\alpha^2+4}{\tau^2+4}}  \frac{\KK(m_1) E(m_2) + m_2 \EE(m_1) K(m_2) - \KK(m_1) K(m_2)}{K(m_2)^2}.
	\]
\end{proposition}

\begin{proof}
	With the help of the integral tables in~\cite{byrd1971}, we obtain the
	following explicit evaluation of the periods.
	\begin{align*}
		(I_1+I_3)(\alpha,\alpha;\tau) &= (J_1+J_3)(\alpha,\alpha;\tau) = \frac{2}{\sqrt{\tau^2+4}}\KK(m_1) \ ,\\
		I_2(\alpha,\alpha;\tau) &= J_2(\alpha,\alpha;\tau) = \frac{\alpha}{\tau}\frac{2}{\sqrt{\alpha^2+4}} K(m_2) \ .
	\end{align*}

	Then we evaluate the derivatives
	\[
		I'_k(\alpha,\alpha;\tau) = \frac{\partial}{\partial \beta} \Bigr|_{\beta=\alpha} I_k(\alpha,\beta;\tau),\qquad
		J'_k(\alpha,\alpha;\tau) = \frac{\partial}{\partial \beta} \Bigr|_{\beta=\alpha} J_k(\alpha,\beta;\tau),
	\]
	and obtain
	\begin{align*}
		(I'_1 + I'_3)(\alpha,\alpha;\tau) &= 0 \ ,\\
		I'_2(\alpha,\alpha;\tau) &= \frac{1}{\tau\sqrt{\alpha^2+4}} K(m_2) \ ,\\
		(J'_1 + J'_3)(\alpha,\alpha;\tau) &= \frac{\alpha}{\tau^2-\alpha^2}\frac{2}{\sqrt{\tau^2+4}}\Big( \KK(m_1) - \frac{\tau^2+4}{\alpha^2+4} \EE(m_1) \Big) \ ,\\
		J'_2(\alpha,\alpha;\tau) &= \frac{2}{\tau \sqrt{\alpha^2+4}} \Big( \frac{\tau^2}{\tau^2-\alpha^2} E(m_2) - \frac{1}{2} K(m_2) \Big) \ .
	\end{align*}

	Finally, by l'H\^opital,

	\begin{align*}
 		&\lim_{a \to b} \tilde Q(\alpha,\beta;\tau)
 		= \frac{\partial Q}{\partial \beta} \Bigr|_{\alpha=\beta}\\
		&= \frac{\tau^2}{\alpha^2}\frac{\tau}{\tau^2-\alpha^2}\sqrt{\frac{\alpha^2+4}{\tau^2+4}}  \frac{\KK(m_1) E(m_2) + m_2 \EE(m_1) K(m_2) - \KK(m_1) K(m_2)}{K(m_2)^2}.
	\end{align*}

	Now note that the function $\tilde Q$ can be extended to a holomorphic
	function of complex arguments $\alpha, \beta, \tau$. The computation above
	shows that it remains bounded for $\alpha=\beta$, and hence extends
	holomorphically to $\alpha=\beta$.  In particular, the extension for real
	arguments is real analytic.
\end{proof}

\begin{figure}[hbt]
	\includegraphics[width=0.48\textwidth]{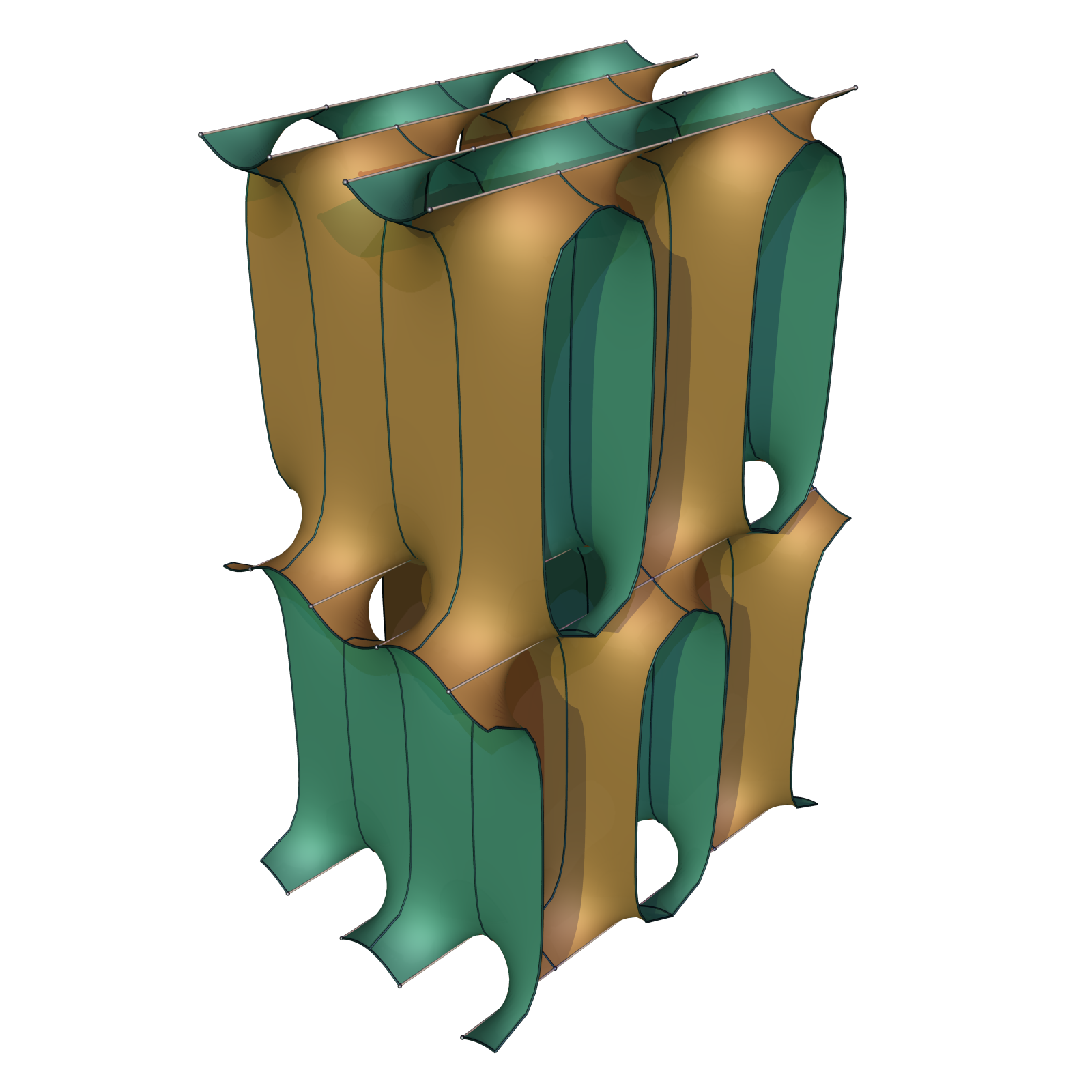}
	\includegraphics[width=0.48\textwidth]{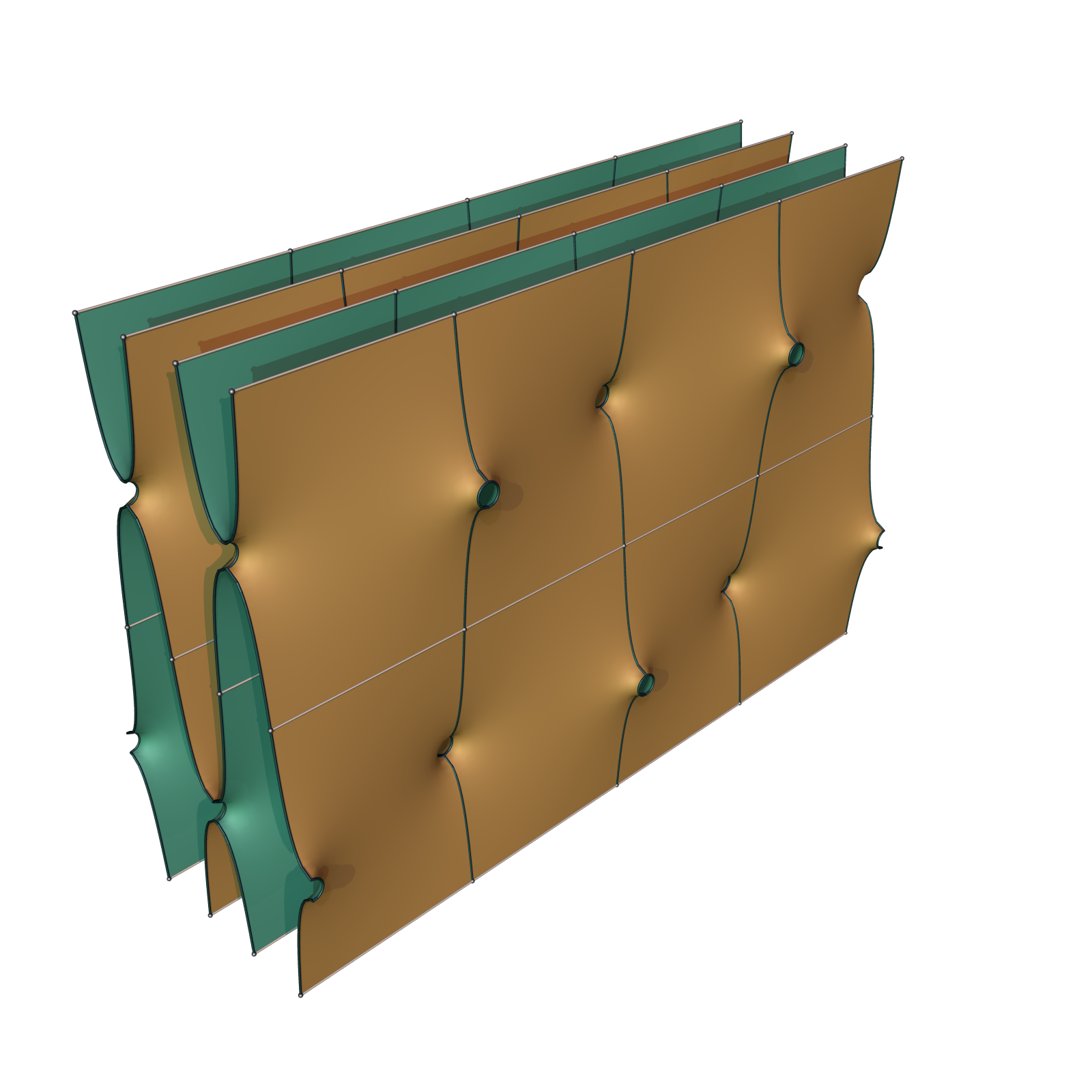}
	\caption{Two surfaces in the intersection of $\overline\oH \cap \oP$}
	\label{fig:intersect}
\end{figure}

The solution set to~\eqref{eq:intersection} is shown in
Figure~\ref{fig:plotintersect}.  In Figure \ref{fig:intersect}, we show two
surfaces in the intersection with extreme values of $\alpha$.  In the next
section, we will analyze the Traizet limit on the right (with small $\alpha$).
The left image strongly suggest that, in the limit of large $\alpha$, the
family tends to a doubly periodic Karcher-Meeks-Rosenberg surface of genus
1~\cite{karcher1988, karcher1989, meeks1989}. 

\section{Revisiting the Traizet limit}
\label{sec:traizet}

With our parametrization of $\oH$, the Traizet limit, with infinitesimally
small catenoid nodes, corresponds to the limit $ab\to 1$ or $\alpha+\beta\to
0$.  In this limit, the Gauss map $G(z)=i$ except at two singular points
$z=1/a=b$ and $z=-a=-1/b$.  So the octagon degenerates into the plane $y=0$ as
expected.  The angle of the limit rhombic torus can be computed as
\begin{equation}\label{eq:rhombicangle}
	\tan\frac{\theta}{2} = \frac{|\int_{-1/t}^{1/t}dh|}{|\int_{1/t}^{t}dh|} =
	\frac{K'(m)}{K(m)},
\end{equation}
where $m=\tau^2/(\tau^2+4)$.

From \eqref{eq:intersection}, we can already locate the Traizet limit of
$\overline\oH \cap \oP$.  First note that $m_2 \to 0$ when $\alpha = \beta \to
0$.  Divide both sides of \eqref{eq:intersection} by $m_2$ to eliminate the
trivial but meaningless solution at $\alpha=\beta=0$.  Recall
that~\cite{byrd1971} $(K(m)-E(m))/m \to \pi/4$ and $K(m) \to \pi/2$ as $m \to
0$.  Hence we obtain for the Traizet limit
\[
	2E(m) = K(m),
\]
where $m = \tau^2/(\tau^2+4)$.  This is uniquely solved with $\tau \approx
4.35932$ or $t \approx 4.57777$.  By putting these parameters
into~\eqref{eq:rhombicangle}, we recover the angle
\[
	\theta^* \approx 1.23409 \approx 70.7083^\circ.
\]

\medskip

While we did not manage to prove uniqueness Conjecture~\ref{conj:unique}, we
can however prove the uniqueness at the Traizet limit.

\begin{theorem}\label{thm:1}
	For any $0 < \theta < \theta^*$, there is a non-trivial solution
	$0<x<1/2$ that solves the balance equation~\eqref{eq:traizet1}. 
	This solution is unique and non-degenerate, hence is the Traizet limit for a
	family of TPMS.
\end{theorem}


\begin{proof}
	Recall that $T_{1,2} = \exp(\pm i \theta/2)$ and $T_3 = -T_1-T_2 =
	-2\cos(\theta/2)$.  We consider the function
	\[
		f(x;\theta) = x \eta_3(\theta) - \zeta(x T_3(\theta);\theta).
	\]
	Observe the following properties of $f$.
	\begin{itemize}
		\item $f \to +\infty$ as $x\to 0+$.  To see this, note that the lattice is
			spanned by conjugate vectors $T_1, T_2$, so $\zeta(z)$ is real for real
			$z$.  By definition, it has residue $+1$ at $0$.  The claim follows.

		\item $\partial_x^2 f>0$ for $0<x<1/2$.  This can be seen by noting that
			\[
				\frac{\partial^2 f}{\partial x^2} = T_3^2 \wp'(x T_3)
			\]
			is clearly non-zero.  So $f$ must be convex in $x$.
	\end{itemize}

	So there is a unique non-trivial solution of $f(x)=0$ if 
	\[
		g(\theta):=\frac{\partial f}{\partial x}\Big|_{x=1/2} = T_3\wp(T_3/2) +
		\eta_3 > 0,
	\]
	or no non-trivial solution otherwise.  These two cases are separated by the zeroes of $g(\theta)$.  They correspond to the Traizet limit of
	$\overline\oH \cap \oP$, which is uniquely given by $\theta^*$.  Since
	non-trivial solutions are known for $\theta = \pi/3 < \theta^*$, we conclude
	that $g(\theta) > 0$ if and only if $0<\theta<\theta^*$.

	For the non-degeneracy, we consider the function
	\[
		F(x,y;\theta) = (x+y) \eta_1(\theta) + (x-y) \eta_2(\theta) -
		\zeta((x+y)T_1(\theta) + (x-y)T_2(\theta);\theta) \ .
	\]
	Note that $F(x,0;\theta) = -f(x,\theta)$.  So the convexity of $f$ in $x$
	implies that $\partial_x F(x,0;\theta)$ is real negative at a non-trivial
	solution of~\eqref{eq:traizet1}.  On the other hand
	\[
		\frac{\partial F}{\partial y}\Bigr|_{y=0} = \eta_1 - \eta_2 + (T_1-T_2) \wp
		(x T_3)
	\]
	is positive purely imaginary for $0<x<1/2$.  This can be seen by noting that
	\[
		\frac{\partial}{\partial x} \frac{\partial F}{\partial y}\Bigr|_{y=0} =
		(T_2^2-T_1^2) \wp' (x T_3)
	\]
	is negative purely imaginary for $0<x<1/2$, and $\partial_y F$ is positive
	purely imaginary at $(x,y)=(1/2,0)$.  The non-degeneracy then follows
	readily.
\end{proof}

Note that, as $\theta$ approaches $\theta^*$, the non-trivial solution tends to
$1/2$, and $T_3 \wp(T_3/2)+\eta_3$ tends to $0$.  So the balanced configuration
at $(x,\theta)=(1/2,\theta^*)$ is degenerate.

\medskip

We may also recover solutions of~\eqref{eq:traizet1} in terms of our
parametrization.

\begin{proposition}
	The Traizet limit of $\oH$ is described by the equation
	\begin{equation}\label{eq:traizet2}
		(2\beta^2-\tau^2+4)K(m) = 2(\beta^2+4)\Pi(n,m),
	\end{equation}
	where $\Pi(n,m)$ is the complete elliptic integral of the third kind, with
	the characteristic $n = \tau^2/(\tau^2-\beta^2)>0$, and the modulus
	$m=\tau^2/(\tau^2+4)$.
\end{proposition}

\begin{figure}[hbt]
	\includegraphics[width=0.4\textwidth]{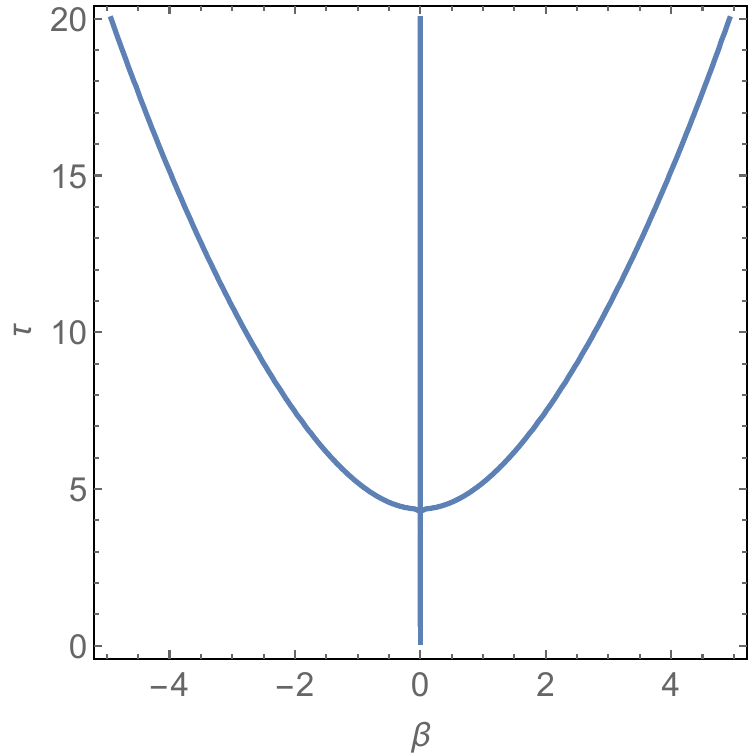}
	\caption{
		Solution set $(\beta,\tau)$ to the period condition \eqref{eq:traizet2}
		describing the Traizet limit of $\oH$, together with the trivial locus $\beta
		= 0$ describing the Traizet limit of $\oP$.  To compare with
		Figure~\ref{fig:diagsols}.
	}
	\label{fig:plottraizet2}
\end{figure}

\begin{proof}
	Again, with the help of the integral tables in \cite{byrd1971}, we obtain
	$I_2(-\beta,\beta;\tau)=J_2(-\beta,\beta;\tau)=0$ and
	\[
		(I_1+I_3)(-\beta,\beta;\tau) =
		(J_1+J_3)(-\beta,\beta;\tau)
		= \sqrt{1-m} K(m) \ ,
	\]
	and the derivatives up to order $2$ with respect to $\alpha$ at $\alpha = -\beta$
	\begin{align*}
		(I'_1 + I'_3)(\alpha,\beta;\tau) = (J'_1+J'_3)(\alpha,\beta;\tau) &= \frac{\beta}{\tau^2-\beta^2} \frac{1}{\sqrt{\tau^2+4}} \Pi(n,m) \ ,\\
		I'_2(\alpha,\beta;\tau) = J'_2(\alpha,\beta;\tau) &= \frac{\pi}{2}\frac{1}{\sqrt{(\tau^2-\beta^2)(\beta^2+4)}} \ ,\\
		I''_2(\alpha,\beta;\tau) = J''_2(\alpha,\beta;\tau)
		&= \frac{\pi}{2}\frac{\beta}{\sqrt{(\tau^2-\beta^2)(\beta^2+4)}}
		\Big(\frac{1}{\tau^2-\beta^2}-\frac{1}{\beta^2+4}\Big).
	\end{align*}

	We look at a modified period condition, namely
	\[
		\hat Q=\frac{1/Q_I-1/Q_J}{(\alpha+\beta)^2} = 0 \ .
	\]
	The evaluations above suffice to compute, by l'H\^opital, that
	\begin{align*}
		&\lim_{\alpha \to -\beta+} \hat Q(\alpha,\beta;\tau)
		= \frac{\partial^2}{\partial \alpha^2} \Bigr|_{\alpha = -\beta} (1/Q_I-1/Q_J)\\
		&= \frac{\pi\beta\sqrt{\tau^2+4}}{4(\beta^2+4)^{3/2}(\tau^2-\beta^2)^{3/2}}\frac{(2\beta^2-\tau^2+4)K(m) - 2(\beta^2+4)\Pi(n,m)}{K(m)^2}.
	\end{align*}

	Hence $\hat Q$ extends analytically to $\alpha + \beta = 0$.  Under the
	constraint $\tau > \beta$, we notice indeed two loci: $\beta=0$ for the
	Traizet limit of $\oP$, and the Traizet limit of $\oH$ must be described by
	\eqref{eq:traizet2}.
\end{proof}

And \eqref{eq:traizet2} must be describing the unique non-trivial locus of
\eqref{eq:traizet1}.  Its solution set is plotted in
Figure~\ref{fig:plottraizet2}.  Alternatively, \eqref{eq:traizet2} can also be
written in the forms
\[
	(\tau^2+4)K(m) = 2(\beta^2+4)\Pi(n',m)
\]
where $n' = (\tau^2-\beta^2)/(\tau^2+4)$, or
\[
	\Big(8\frac{\tau^2}{n''}+(\tau^2-4)(\beta^2+4)\Big)K(m) = 8\Big(\frac{\tau^2}{n''}-(\beta^2+4)\Big)\Pi(n'',m)
\]
where $n'' = \beta^2/(\beta^2+4)$.

To find the intersection with the trivial locus, let $\beta \to 0+$.  For the
three forms of~\eqref{eq:traizet2}, we recall, respectively, that
\begin{align*}
	\lim_{n\to 1+} \Pi(n,m) &= K(m) - \frac{E(m)}{1-m} & \text{\cite[(19.6,6)]{DLMF};}\\
	\lim_{n'\to m} \Pi(n',m) &= \frac{E(m)}{1-m} & \text{\cite[(19.6.1)]{DLMF};}\\
	\lim_{n''\to 0+} \frac{\Pi(n'',m)-K(m)}{n''} &= \frac{K(m)-E(m)}{m} & \text{cf.\ \cite[(733.00)]{byrd1971}.}
\end{align*}
Any one of these leads once again to
\[
	2E(m) = K(m).
\]

\begin{remark}
	The magic equation $2E(m) = K(m)$ also appeared in~\cite{chenweber1} for
	locating the bifurcation point in the $\tD$ family.
\end{remark}

\begin{remark}
	Assume that the limit torus is spanned by $T_1=1$ and $T_2=\tau$.  We have
	studied the rhombic case $|\tau|=1$.  Numerically, we find that if $\tau$ is
	taken from the colored region on the left in Figure~\ref{fig:general}, within
	the fundamental domain of the modular group, then there is unique non-trivial
	position $p_2(\tau)$ that solves Traizet's balance
	equation~\eqref{eq:balanced}.  The left boundary curve of this region, which
	is asymptotic to the circle $|z-1|=1$ as $|\tau|\to 0$, represents a
	one-parameter family of tori for which the trivial configuration $x=y=1/2$ is
	degenerate.  The image of the continuous map $\tau \mapsto p_2(\tau)$ is the
	colored region on the right in Figure~\ref{fig:general}.  The coloring should
	help to visualize the map.
\end{remark}

\begin{figure}[hbt]
	\includegraphics[width=0.4\textwidth]{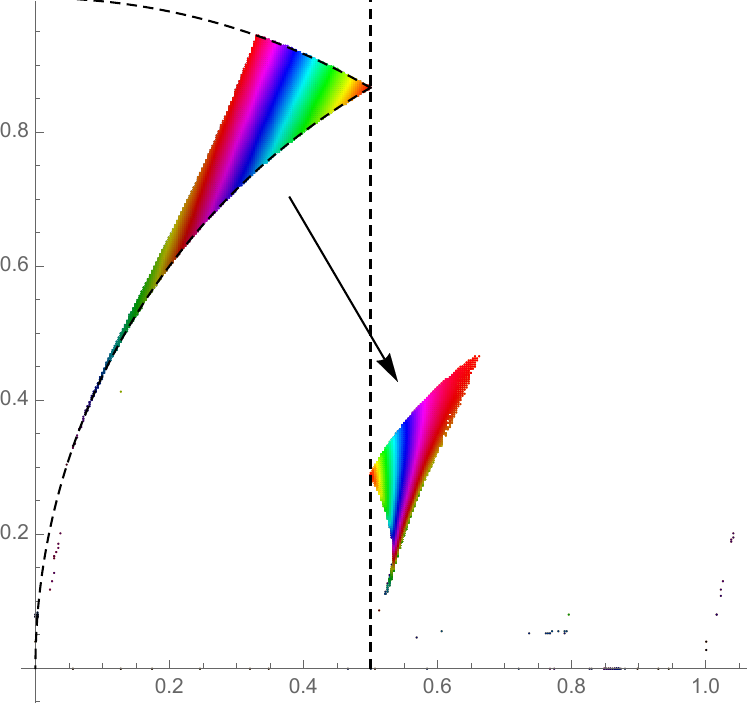}
	\caption{\label{fig:general}}
\end{figure}

\begin{remark}
	It recently comes to our attention that the solutions of Traizet's balance
	equation~\eqref{eq:balanced} have been systematically studied in the PDE
	contexts as the critical points of the Green function on flat
	tori~\cite{lin2010, chenzj2018, bergweiler2016}.  For a fixed torus, apart
	from the trivial solutions at the 2-division points, there could be at most
	one pair of non-trivial solutions.  In other words, the balancing equation
	has either three or five solutions.  The boundary between the two cases
	provides a 1-parameter family of degenerate balanced configuration.
\end{remark}

\medskip

Let $k_r$ denote the $r$-th \emph{elliptic integral singular values}, i.e.\
$K'(k_r^2)/K(k_r^2)=\sqrt{r}$.  A table of $k_r$ can be found
in~\cite[p.~95]{bowman1961} and \cite[(4.6.10)]{borwein1987}.

It was calculated by Legendre (see~\cite[\S 22.81]{whittaker1962}) that $k_3 =
(\sqrt{6}-\sqrt{2})/4$, hence $K'(m)/K(m) = 1/\sqrt{3}$ when $m = 1-k_3^2 =
(2+\sqrt{3})/4$.  We then see from~\eqref{eq:rhombicangle} that the rhombic
torus with $\theta = 60^\circ$ occurs when $\tau = 2(2+\sqrt{3})$.
Then~\eqref{eq:traizet2} is solved, very conveniently, with $\beta = 2$.  One
then verifies that the singular point at $\beta$ is mapped to one third of the
height of the box.  These are then explicit parameters for the Traizet limit of
Schwarz' $\H$ family.

We are now ready to prove:

\begin{theorem} 
	Schwarz $\H$ surfaces can be deformed within the set of TPMS of genus three
	into Meeks surfaces.
\end{theorem}

\begin{proof} 
	Within a sufficiently small neighborhood of a Traizet limit, Traizet's
	construction actually implies a homeomorphism between the space of TPMS of
	genus three and the space of $3$-tori.  This was not explicitly stated
	in~\cite{traizet2008}, but follows from his design of the Weierstrass data
	and the uniqueness in the implicit function theorem, as argued
	in~\cite{traizet2002}.  Let $U$ be such a neighborbood of the Traizet limit
	of $H$.  In particular, $U \cap \oH$ is connected.

	Now fix $\epsilon > 0$.  We consider the $\oH$ surfaces with $\alpha + \beta
	= \epsilon$.  The period condition for such surfaces is $\tilde
	Q_\epsilon(\beta,\tau) = \tilde Q(\epsilon-\beta, \beta; \tau) = 0$, defined
	on the region $\{(\beta, \tau) \in \R_+^2 \colon \tau > \beta \ge
	\epsilon/2\}$.  We have shown that $\tilde Q_\epsilon(\beta; \tau)$ is
	negative as $\tau$ approaches $\beta$, and positive as $\tau$ tends to
	infinity.  This holds, in particular, also for $\alpha = \beta = \epsilon/2$.
	Hence in the real analytic solution set of $\tilde Q_\epsilon = 0$, there
	must be a real analytic curve $\gamma$ that separates the line $\tau = \beta$
	from $\tau = \infty$.  If $\epsilon$ is sufficiently small, the curve
	$\gamma$ passes through $U$.

	So we deform an $\H$ surface first along the $\H$ family into $U$, then
	within $U \cap \oH$ onto the curve $\gamma$, finally along $\gamma$ until an
	$\oP$ surface.  The latter belongs to Meeks, which is connected.  Note that
	this deformation path is within $\oH$ until hitting $\oP$. 
\end{proof}

\begin{remark}
	It is easy to find $k_1=1/\sqrt{2}$, hence $K'(m)/K(m)=1$ when $m = 1-k_1^2 =
	1/2$.  We then see from~\eqref{eq:rhombicangle} that the rhombic torus
	becomes square when $\tau = 2$.  In this case, \eqref{eq:traizet2} has no
	solution with $\beta < \tau$.  So the only balanced configuration is with
	$\beta = 0$.  This is the Traizet limit of the tetragonal deformation family
	$\tP$ of Schwarz' $\P$ surface.
\end{remark}

\bibliography{References}
\bibliographystyle{alpha}

\end{document}